\documentclass[11pt]{article}
\usepackage{amssymb}
\usepackage[perpage,symbol]{footmisc}
\usepackage{listings}
\usepackage{amsfonts}
\usepackage{enumerate}
\usepackage{amsmath}
\usepackage{cite}
\usepackage{array}
\usepackage{booktabs}

\begin{document}

\newtheorem{theorem}{Theorem}[section]
\newtheorem{corollary}[theorem]{Corollary}
\newtheorem{definition}[theorem]{Definition}
\newtheorem{proposition}[theorem]{Proposition}
\newtheorem{lemma}[theorem]{Lemma}
\newtheorem{conjecture}[theorem]{Conjecture}
\newenvironment{proof}{\noindent {\bf Proof.}}{\rule{3mm}{3mm}\par\medskip}
\newcommand{\remark}{\medskip\par\noindent {\bf Remark.~~}}

\title{Incidence Matrices of  Polarized Projective Spaces\footnote{This work is supported by the National Natural Science Foundation of China (No. 11071160).}}
\author{Chunlei Liu\footnote{Dept. of Math., Shanghai Jiaotong Univ., Shanghai,
200240, clliu@sjtu.edu.cn.}, Haode Yan\footnote{Corresponding author,
Dept. of Math., SJTU, Shanghai, 200240, hdyan@sjtu.edu.cn.}}

\maketitle
\thispagestyle{empty}

\abstract{In this paper, we first define  a non-degenerate symmetric bilinear form on $\mathbb{F}_{q}^4$. Then we get an incidence  matrix $\mathbf{G}$ of  $\mathbb{F}_{q}^4$ by the bilinear form. By its corresponding quadratic form $Q$, the lines of $\mathbb{F}_{q}^4$ are classified as isotropic and anisotropic lines. Under this classification, we can get two sub-matrices of $\mathbf{G}$ and prove their 2-rank.}

\noindent {\bf Key words and phrases}: finite field, quadratic form, incidence matrix.
\section{\small{INTRODUCTION}}
Throughout this paper, $q$ is an odd prime power, and
$\mathbb{F}_{q}$ is the finite field with $q$ elements. We endow the
space $V=\mathbb{F}_{q}^4$ with a polarization, i.e., a
non-degenerate symmetric bilinear form $\langle \  ,\ \rangle$ over
$\mathbb{F}_q$. Let $Q$ be the corresponding quadratic form, without
loss of generality, we always assume that
$Q(X)=x_0^2-x_1^2+x_2^2-\alpha x_3^2$, where $\alpha$ is a nonzero
element in $\mathbb{F}_{q}$. As usual, we denote by $\mathbf{P}(V)$ the
projective space of $V$, which is formed of the 1-dimensional
subspaces of $V$. For each nonzero vector $X\in V$, let $\bar{X}\in
\mathbf{P}(V)$ be the 1-dimensional subspace generated by $X$. The quadratic
form $Q$ gives rise to a matrix whose rows and columns are indexed
by elements of $\mathbf{P}(V)$. This matrix is defined by the formula
\[\mathbf{G}:\mathbf{P}(V)\times \mathbf{P}(V) \rightarrow \{ 0, 1 \},
\mathbf{G}(\bar{X},\bar{Y})=
\begin{cases}1,\bar{X} \perp {\bar{Y}},\\0,\bar{X} \not\perp {\bar{Y}}. \end{cases}
\]
We call that matrix the incidence matrix of $\mathbf{P}(V)$ associated to
$Q$.

A vector $\bar{X}\in \mathbf{P}(V)$ is called {\it isotropic} or {\it
anisotropic} according to whether $X\in V$ is {\it isotropic} or
{\it anisotropic}. Let $I$ (resp. $A$) be the set of isotropic
(resp. anisotropic) vectors in $\mathbf{P}(V)$, and set \[\mathbf{G}_{II}=
\mathbf{G}|_{I\times I}\text{ and }
\mathbf{G}_{AA}=\mathbf{G}|_{A\times A}. \]
From \cite{1,3}, we have that the 2-rank of $\mathbf{G}_{II}$ is $q+1$ and the 2-rank of $\mathbf{G}_{AA}$ is $q^{2}-1$, respectively, in the case of $V=\mathbb{F}_{q}^3$.
Moreover, based on computational evidence, it was conjectured in \cite{1}
that in the case of $V=\mathbb{F}_{q}^3$, $\mathbf{G}_{II}$ is always full rank and $\mathbf{G}_{AA}$ is  full rank or its 2-rank is one less than its order  according to $n$ is odd or even. In this paper, we will prove the following theorem:
\begin{theorem}
In the case of $V=\mathbb{F}_{q}^4$, $\mathbf{G}_{II}$ and $\mathbf{G}_{AA}$  are of full rank.
\end{theorem}

\section{The isotropic case}
\begin{proposition}$\mathbf{{G}}_{II}$ is of full rank over $\mathbb{F}_2$.
\end{proposition}

\begin{lemma}\label{Prop:1.1}
If $\bar{X},\bar{Y}$ are different isotropic lines such that $\langle X,Y \rangle=0$.Then any line on the plane that generated by $\bar{X}$ and $\bar{Y}$  is isotropic.
\end{lemma}
\begin{proof}
The lines on that plane is in $\mathbf{P}(\mathbb{F}_{q} X+\mathbb{F}_{q} Y)$.Because $\bar{X}$,$\bar{Y}$ are isotropic lines and $\langle X,Y\rangle=0$. Consider with the bilinear form,we get  $\langle k_1 X+k_2Y,k_1 X+k_2 Y\rangle=0$, any $k_1,k_2 \in \mathbb{F}_{q}$.Lemma has proved.
\end{proof}

\begin{lemma}
If there are three different isotropic lines on a plane,then any line on that plane is isotropic.
\end{lemma}
\begin{proof}
Suppose $\bar{X},\bar{Y},\bar{Z}$ are different isotropic lines on a common plane.Then $\bar{Z}$ is an element in $\mathbf{P}(\mathbb{F}_{q} X+\mathbb{F}_{q} Y)$.From $\langle Z,Z\rangle=0$,we get $\langle X,Y\rangle=0$.From prop1.1,we can get the result.
\end{proof}

From the bilinear form, consider the lines in the orthogonal complement of $\bar{X}$. We use ${\bar{X}}^\perp$ to represent the set of the lines in this hyperplane,it is ${\bar{X}}^\perp=\{\bar{W}|\langle W,X \rangle =0\}$.Then we get following lemma.

\begin{lemma}$\bar{X}$, $\bar{Y}$, $\bar{Z}$ are different lines,and ${\bar{X}}^\perp \cap  {\bar{Y}}^\perp={\bar{X}}^\perp \cap {\bar{Y}}^\perp \cap {\bar{Z}}^\perp$.Then $\bar{Z}$ is in $\mathbf{P}(\mathbb{F}_{q} X+\mathbb{F}_{q} Y)$.
\end{lemma}
\begin{proof}From the question we can get two groups of equations.Because  they have same solutions, so the rank of coefficient matrix must be same. $\bar{X}$,$\bar{Y}$ are different lines, so the rank of the first coefficient matrix is 2, so is the second coefficient matrix.Then $Z$ can be linear represented as $k_1 X+k_2 Y$, $k_1,k_2\in\mathbb{F}_{q} $, lemma has proved.
\end{proof}

\begin{lemma}(1)let $\alpha$ be a nonzero square element of $\mathbb{F}_{q}$, that is $Q(X)=x_{0}^{2}-x_{1}^{2}+x_{2}^{2}-x_{3}^{2}$.
We can get:any isotropic line $\bar{X}$, there are $2q+1$ isotropic lines and $q^2-q$ anisotropic lines on ${\bar{X}}^{\perp}$;any anisotropic line $\bar{Y}$,there are $q+1$ isotropic lines and $q^2$ anisotropic lines on ${\bar{Y}}^{\perp}$.\\
(2)let $\alpha$ be a non-square element of $\mathbb{F}_{q}$, that is $Q(X)=x_{0}^{2}-x_{1}^{2}+x_{2}^{2}-\alpha x_{3}^{2}$.
We can get:any isotropic line $\bar{X}$,there is $1$ isotropic line and $q^2+q$ anisotropic lines on ${\bar{X}}^{\perp}$;any anisotropic line $\bar{Y}$,there are $q+1$ isotropic lines and $q^2$ anisotropic lines on ${\bar{Y}}^{\perp}$.
\end{lemma}
\begin{proof}According to \cite{1}(lemma 1.9),we know if $\alpha$ is a nonzero square element,there are $q^2+2q+1$ isotropic lines and $q^3-q$ anisotropic lines on $\mathbb{F}_{q}^4$.If $\alpha$ is a non-square element,there are $q^2+1$ isotropic lines and $q^3+q$ anisotropic lines on $\mathbb{F}_{q}^4$.\\
(1)First we prove there are $2q+1$ isotropic lines on ${\bar{X}}^{\perp}$. $\langle X,X\rangle=0$,we want to find line ${\bar{Z}}$ such that $\langle X,Z\rangle=0$ and $\langle Z,Z\rangle=0$ set up at the same time.Let $X=(x_{0},x_{1},x_{2},x_{3})$,discuss below:\\
1.$x_{0}^2-x_{1}^2\not=0$.\\
So,we get $x_{0}^2-x_{1}^2=x_{3}^2-x_{2}^2\not=0$.Consider $P=(x_{3}-x_{2},x_{3}-x_{2},x_{0}-x_{1},x_{0}-x_{1})$, $Q=(x_{3}+x_{2},x_{3}+x_{2},-x_{0}+x_{1},x_{0}-x_{1})$ are the solutions of equations.We can verify that $P,Q\not=(0,0,0,0)$ and $\bar{X}$,$\bar{P}$,$\bar{Q}$ are different.Now,we get two different isotropic lines $\bar{P}$, $\bar{Q}$ on ${\bar{X}}^{\perp}$,and $\langle P,Q\rangle=-2(x_{0}^2-x_{1}^2)\not=0$.We know $\bar{X}$, $\bar{P}$, $\bar{Q}$ are not in a common plane,otherwise,according to lemma 2.2,lines on the plane are all isotropic.In particular,line  $\overline{P+Q}$ is an isotropic line.We can get $\langle P,Q\rangle=0$ by the bilinear form and $\langle P,P \rangle=\langle Q,Q \rangle =0$,contradiction.Because $\bar{X},\bar{P}$ are isotropic lines and $\langle X,P\rangle=0$,according to lemma 2.2,the lines on the plane that generated by $\bar{X}$ and $\bar{P}$ are all isotropic.It means the elements in $\mathbf{P}(\mathbb{F}_{q} X+\mathbb{F}_{q} P)$ are isotropic.The same as plane generated by $\bar{X},\bar{Q}$,we get the elements in $\mathbf{P}(\mathbb{F}_{q} X+\mathbb{F}_{q} Q)$ are isotropic.We can calculate each of $\mathbf{P}(\mathbb{F}_{q} X+\mathbb{F}_{q} P)$ and $\mathbf{P}(\mathbb{F}_{q} X+\mathbb{F}_{q} Q)$ has $q+1$ elements,and $\bar{X}$ is the unique element in both of them.So,we get $2q+1$ different isotropic lines in ${\bar{X}}^{\perp}$.Let's proof there isn't any other isotropic line on ${\bar{X}}^{\perp}$.Regard $\bar{X}$,$\bar{P}$,$\bar{Q}$ as basis of ${\bar{X}}^{\perp}$,the lines on ${\bar{X}}^{\perp}$ can be expressed as
$\mathbf{P}(\mathbb{F}_{q} X+\mathbb{F}_{q} P+\mathbb{F}_{q} Q)$.If there is another isotropic line on ${\bar{X}}^{\perp}$,let it be $\bar{R}$, $R=k_1 X+k_2 P+k_3 Q$,according to $\langle R,R\rangle=0$,we can get $k_{2}=0$ or $k_3 =0$,R is in $\mathbf{P}(\mathbb{F}_{q} X+\mathbb{F}_{q} P)$ or $\mathbf{P}(\mathbb{F}_{q} X+\mathbb{F}_{q} Q)$,contradiction.So,in this case,there are $2q+1$ isotropic lines on ${\bar{X}}^{\perp}$.\\
2.$x_{0}^2-x_{1}^2=0$,but $x_{3}^2-x_{0}^2\not=0$.\\
Let $\dot{x_{0}}=x_{3}$, $\dot{x_{1}}=x_{0}$, $\dot{x_{2}}=x_{1}$, $\dot{x_{3}}=x_{2}$.According to $x_{0}^2-x_{1}^2+x_{2}^2-x_{3}^2=0$,we can get ${\dot{x_{0}}}^2-{\dot{x_{1}}}^2+{\dot{x_{2}}}^2-{\dot{x_{3}}}^2=0$,and ${\dot{x_{0}}}^2-{\dot{x_{1}}}^2\not=0$.It's same as case 1.\\
3.$x_{0}^2-x_{1}^2=0$ and $x_{3}^2-x_{0}^2=0$.\\
It means $x_{0}^2=x_{1}^2=x_{2}^2=x_{3}^2$.Then $\bar{X}$ is one of the lines below:$\overline{(1,1,1,1)}$,$\overline{(1,1,1,-1)}$,\\$\overline{(1,1,-1,1)}$,
$\overline{(1,1,-1,-1)}$,$\overline{(1,-1,1,1)}$, $\overline{(1,-1,1,-1)}$,$\overline{(1,-1,-1,1)}$,$\overline{(1,-1,-1,-1)}$.
Similar to the proof above,any $\bar{X}$,if we find out two distinct isotropic line $\bar{P}$,$\bar{Q}$,and $\langle P,Q\rangle\not=0$,then we can prove there are $2q+1$ isotropic lines on ${\bar{X}}^{\perp}$.\\
If $\bar{X}=\overline{(1,1,1,1)}$, let $\bar{P}=\overline{(1,1,-1,-1)}$, $\bar{Q}=\overline{(1,-1,-1,1)}$.\\
If $\bar{X}=\overline{(1,1,1,-1)}$, let $\bar{P}=\overline{(1,1,-1,1)}$, $\bar{Q}=\overline{(1,-1,-1,-1)}$.\\
If $\bar{X}=\overline{(1,1,-1,1)}$, let $\bar{P}=\overline{(1,1,1,-1)}$, $\bar{Q}=\overline{(1,-1,1,1)}$.\\
If $\bar{X}=\overline{(1,1,-1,-1)}$, let $\bar{P}=\overline{(1,1,1,1)}$, $\bar{Q}=\overline{(1,-1,1,-1)}$.\\
If $\bar{X}=\overline{(1,-1,1,1)}$, let $\bar{P}=\overline{(1,-1,-1,-1)}$, $\bar{Q}=\overline{(1,1,-1,1)}$.\\
If $\bar{X}=\overline{(1,-1,1,-1)}$, let $\bar{P}=\overline{(1,-1,-1,1)}$, $\bar{Q}=\overline{(1,1,-1,-1)}$.\\
If $\bar{X}=\overline{(1,-1,-1,1)}$, let $\bar{P}=\overline{(1,-1,1,-1)}$, $\bar{Q}=\overline{(1,1,1,1)}$.\\
If $\bar{X}=\overline{(1,-1,-1,-1)}$, let $\bar{P}=\overline{(1,-1,1,1)}$, $\bar{Q}=\overline{(1,1,1,-1)}$.\\
In conclusion,we proved when $\alpha$ is a nonzero square element, any isotropic line $\bar{X}$ on $\mathbb{F}_{q}^4$,there are $2q+1$ isotropic lines on ${\bar{X}}^{\perp}$.Because there are $q^2+q+1$ lines on ${\bar{X}}^{\perp}$, so there are $q^2-q$ anisotropic lines on ${\bar{X}}^{\perp}$.If $\bar{Y}$ is an anisotropic line , we can calculate that the number of isotropic lines on ${\bar{Y}}^{\perp}$ equals $\frac{(q+1)^2 (q^2-q)}{q^3-q}$,it is $q+1$.The number of anisotropic lines on ${\bar{Y}}^{\perp}$ is $q^2$.(1) has proved.\\
(2)Now $\alpha$ is a non-square element and $Q(X)=x_{0}^{2}-x_{1}^{2}+x_{2}^{2}-\alpha x_{3}^{2}$.If $X=(x_{0},x_{1},x_{2},x_{3})$ and line $\bar{X}$ is an isotropic line,we want to prove $\bar{X}$ is the unique isotropic line on ${\bar{X}}^{\perp}$.Discuss below:\\
$1^{'}$. $x_{2}^2\not=\alpha x_{3}^2$.It is said $x_{0}^2-x_{1}^2=x_{2}^2-\alpha x_{3}^2\not=0$. ${\bar{X}}\in {\bar{X}}^{\perp}$,and we can find out $P=(x_{1},x_{0},\alpha x_{3}, x_{2})$, $Q=(-x_{1},-x_{0},\alpha x_{3},x_{2})$. $\bar{P},\bar{Q}\in{\bar{X}}^{\perp}$.We can verify that $\bar{P},\bar{Q}$ are different lines and $\bar{X},\bar{P},\bar{Q}$ are not in a common plane.Notice that $\langle P,P\rangle=\langle Q,Q\rangle=(1-\alpha)(x_{2}^2-\alpha x_{3}^2)\not=0$.So $\bar{P},\bar{Q}$ are anisotropic lines.Consider if there exsits an isotropic line $\bar{M} \in \mathbf{P}(\mathbb{F}_{q} P+\mathbb{F}_{q} Q)$.Let $M=k_1 P+k_2 Q$,obviously $k_1,k_2\not=0$.From $\langle M,M \rangle =0$,we get $\langle k_1 P+k_2 Q,k_1 P+k_2 Q\rangle={k_1}^2 \langle P,P\rangle+2k_1 k_2\langle P,Q\rangle+k_2^2\langle Q,Q\rangle=0$.Regard this equation as a quadratic equation about $k_1$. $\Delta=16\alpha k_2^2 (x_2^2-\alpha x_3^2)$ is a non-square element,so the equation has no solution.It means there is no isotropic line in $\mathbf{P}(\mathbb{F}_{q} P+\mathbb{F}_{q} Q)$.Any $\bar{M}\in \mathbf{P}(\mathbb{F}_{q} P+\mathbb{F}_{q} Q)$, $\langle M,M\rangle\not=0$,but $\langle X,M\rangle=k_1\langle X,P\rangle+k_2\langle X,Q\rangle=0$.Consider $\mathbf{P}(\mathbb{F}_{q} X+\mathbb{F}_{q} M)$,any $\overline{t_1 X + t_2 M} \in \mathbf{P}(\mathbb{F}_{q} X+\mathbb{F}_{q} M)$,if $t_2\not=0$, $\langle t_1 X+t_2 M,t_1 X+t_2 M\rangle=t_1^2\langle X,X\rangle+2t_1 t_2\langle M,X\rangle+t_2^2\langle M,M\rangle=t_2^2\langle M,M\rangle\not=0$,so $\bar{X}$ is the unique isotropic line in $\mathbf{P}(\mathbb{F}_{q} X+\mathbb{F}_{q} M)$.And $\bar{M}$ is any line in  $\mathbf{P}(\mathbb{F}_{q} P+\mathbb{F}_{q} Q)$,there are $q+1$ kinds of choices.So we can get $q+1$ planes,every of them has one isotropic line $\bar{X}$ and $q$ anisotropic lines.Every two planes has only one common line,it is $\bar{X}$.So we get $q(q+1)$ different anisotropic lines on ${\bar{X}}^{\perp}$.Because $\bar{X}^{\perp}$ has $q^2+q+1$ lines,so they are the $q(q+1)$ different anisotropic lines and $\bar{X}$.Case $1^{'}$ has proved.\\
$2^{'}$. $x_{2}^2=\alpha x_{3}^2$.Because $\alpha$ is a non-square element,so $x_{2}=x_{3}=0$.Then we know $x_{0}=\pm x_{1}$.So there are only two lines satisfy this condition, $\overline{(1,1,0,0)}$ and $\overline{(1,-1,0,0)}$.Solve the equations we can get there is only one isotropic line $\overline{(1,1,0,0)}$ and $q^2+q$ anisotropic lines on $\overline{(1,1,0,0)}^{\perp}$,there is only one isotropic line $\overline{(1,-1,0,0)}$ and $q^2+q$ anisotropic lines on $\overline{(1,-1,0,0)}^{\perp}$.\\
In conclusion,we proved when $\alpha$ is a non-square element,any isotropic line $\bar{X}$ in $\mathbb{F}_{q}^4$,there is $1$ isotropic line and $q^2+q$ anisotropic lines on ${\bar{X}}^{\perp}$.If $\bar{Y}$ is an anisotropic line,we can calculate the number of isotropic lines on ${\bar{Y}}^{\perp}$ equals $\frac{(q^2+1)(q^2+q)}{q^3+q}$,it is equals $q+1$ , and the number of anisotropic lines on ${\bar{Y}}^{\perp}$ is $q^2$.(2) has proved.
\end{proof}

\begin{lemma}$\bar{P},\bar{Q}$ are different isotropic lines,then the number of isotropic lines on ${\bar{P}}^{\perp}\cap {\bar{Q}}^{\perp}$ is even.
\end{lemma}
\begin{proof}If there is no isotropic line on ${\bar{P}}^{\perp}\cap {\bar{Q}}^{\perp}$, the number must be $0$,even.
If ${\bar{P}}^{\perp}\cap {\bar{Q}}^{\perp}$ has some isotropic lines, then select $\bar{X}$,so $\langle X,X\rangle=0$,select $\bar{Y}\in {\bar{P}}^{\perp}\cap {\bar{Q}}^{\perp},\bar{Y}\not=\bar{X}$. ${\bar{P}}^{\perp}\cap {\bar{Q}}^{\perp}$ is a 2-dimensional subspace of $V$,so it can be generated by $\bar{X},\bar{Y}$.First we want to prove $\langle X,Y\rangle=0$ and $\langle Y,Y\rangle\not=0$ can't set up at the same time.Otherwise, $\langle X,Y\rangle=0$ and $\langle Y,Y\rangle\not=0$.So $\bar{X},\bar{Y}\in {\bar{P}}^{\perp}\cap {\bar{Q}}^{\perp}\cap {\bar{X}}^{\perp}$, ${\bar{P}}^{\perp}\cap {\bar{Q}}^{\perp} \subset {\bar{P}}^{\perp}\cap {\bar{Q}}^{\perp}\cap {\bar{X}}^{\perp}$.It means ${\bar{P}}^{\perp}\cap {\bar{Q}}^{\perp}={\bar{P}}^{\perp}\cap {\bar{Q}}^{\perp}\cap {\bar{X}}^{\perp}$ According to lemma 2.4,we know $X=t_1 P+t_2 Q$,$t_1,t_2 \in \mathbb{F}_{q}$.We can always obtain $\langle P,Q\rangle=0$ from $\langle X,X\rangle=0$.Now ${\bar{P}}^{\perp}\cap {\bar{Q}}^{\perp}=\mathbf{P}(\mathbb{F}_{q} P+\mathbb{F}_{q} Q)$.From lemma 2.2,lines in $\mathbf{P}(\mathbb{F}_{q} P+\mathbb{F}_{q} Q)$ are all isotropic lines.In particular,$\bar{Y} \in {\bar{P}}^{\perp}\cap {\bar{Q}}^{\perp}$,so $\langle Y,Y\rangle=0$,contradiction.It means $\langle X,Y\rangle=0$ and $\langle Y,Y\rangle\not=0$ can't set up at the same time.

Lines on ${\bar{P}}^{\perp}\cap {\bar{Q}}^{\perp}$ are in $\mathbf{P}(\mathbb{F}_{q} X+\mathbb{F}_{q} Y)$.Count the number of isotropic lines in $\mathbf{P}(\mathbb{F}_{q} X+\mathbb{F}_{q} Y)-\{\bar{X}\}$.We must solve equation $\langle t_1 X+t_2 Y,t_1 X+t_2 Y\rangle=0$,$t_2\not=0$.Without loss of generality,let $t_2=1$.It is $2t_1\langle X,Y\rangle+\langle Y,Y\rangle=0$.Discuss below:\\
case 1: $\langle X,Y\rangle=0$,then $\langle Y,Y\rangle=0$ must be set up.At this time,$t_1$ has $q$ solutions.There are $q+1$ isotropic lines on ${\bar{P}}^{\perp}\cap {\bar{Q}}^{\perp}$.Even.\\
case 2: $\langle X,Y\rangle\not=0$.At this time,$t_1$ has only 1 solution.Then the number is 2,even.
Lemma has proved.
\end{proof}

Now we can prove prop 2.1.\\
1.$\alpha$ is a nonzero square element.We calculate $\mathbf{{G}}_{II}^2$.The element at position $(\bar{X},\bar{X})$ of $\mathbf{{G}}_{II}^2$ equals the number of isotropic lines on ${\bar{X}}^{\perp}$.According to lemma 2.5(1),it is $2q+1$,odd.The element at position $(\bar{X},\bar{Y})$ $(\bar{X}\not=\bar{Y})$ of $\mathbf{{G}}_{II}^2$ is the number of isotropic lines on ${\bar{X}}^{\perp} \cap {\bar{Y}}^{\perp}$.According to lemma 2.6,it is even.So,$\mathbf{{G}}_{II}^2 \equiv I \pmod 2$.So,$\mathbf{{G}}_{II}$ is of full rank.\\
2.$\alpha$ is a non-square element.Accoding to lemma 2.5(2),the unique isotropic line on ${\bar{X}}^{\perp}$ is $\bar{X}$.So,$\mathbf{{G}}_{II}=I$.Full rank.The proposition has proved.

\section{The anisotropic case}

\begin{proposition}The 2-rank of $\mathbf{G}_{AA}$ is of full rank over $\mathbb{F}_2$.
\end{proposition}

\begin{lemma} Suppose $\bar{P_1}$, $\bar{P_2}$ are different anisotropic lines.Then there is only one isotropic line on ${\bar{P_1}}^{\perp} \cap {\bar{P_2}}^{\perp}$ if and only if $\langle P_1,P_1\rangle\langle P_2,P_2\rangle=\langle P_1,P_2\rangle^2$.
\end{lemma}
\begin{proof}\\
$\Rightarrow$.If ${\bar{P_1}}^{\perp} \cap {\bar{P_2}}^{\perp}$ has only one isotropic line $\bar{Z}$,then,any other $\bar{X} \in {\bar{P_1}}^{\perp} \cap {\bar{P_2}}^{\perp}$, we have $\langle X,X\rangle\not=0$. $\bar{Z}$ is the unique,so $\langle X+kZ,X+kZ\rangle\not=0$, $k\in \mathbb{F}_{q}$.It means $\langle X,X\rangle+2k\langle X,Z\rangle=0$ has no solutions.So $\langle X,Z\rangle=0$, $\bar{X} \in {\bar{Z}}^{\perp}$.And $\bar{Z} \in {\bar{Z}}^{\perp}$,so we get ${\bar{P_1}}^{\perp} \cap {\bar{P_2}}^{\perp} \subset {\bar{Z}}^{\perp}$, it is ${\bar{P_1}}^{\perp} \cap {\bar{P_2}}^{\perp} \cap {\bar{Z}}^{\perp} = {\bar{P_1}}^{\perp} \cap {\bar{P_2}}^{\perp}$.According to lemma 2.4, $\bar{Z}$ is on the plane that generated by $\bar{P_1},\bar{P_2}$.Because $\bar{Z} \not= \bar{P_1},\bar{P_2}$,so we can suppose $Z=t_1 P_1+t_2 P_2$,$t_1,t_2\in \mathbb{F}_{q}^*$.From $\bar{Z} \in{\bar{P_1}}^{\perp}$,we know $\langle P_1,Z\rangle=0$,it is $t_1\langle P_1,P_1\rangle+t_2\langle P_1,P_2\rangle=0$.From $\bar{Z} \in{\bar{P_2}}^{\perp}$,we know $\langle P_2,Z\rangle=0$,it is $t_1\langle P_1,P_2\rangle+t_2\langle P_2,P_2\rangle=0$. $t_1,t_2\not=0$,we can get $\langle P_1,P_1\rangle\langle P_2,P_2\rangle=\langle P_1,P_2\rangle^2$.\\
$\Leftarrow$.If $\langle P_1,P_1\rangle \langle P_2,P_2\rangle =\langle P_1,P_2\rangle ^2$,so $\langle P_1,P_2 \rangle \not=0$.Let $\frac{\langle P_1,P_1\rangle}{\langle P_2,P_2\rangle}=\frac{\langle P_1,P_2\rangle}{\langle P_2,P_2\rangle}=t_0\in\mathbb{F}_{q}^*$.We can obtain $\langle P_1,P_1-t_0 P_2 \rangle=\langle P_2,P_1-t_0 P_2 \rangle=0$.So $\overline{P_1-t_0 P_2} \in {\bar{P_1}}^{\perp} \cap {\bar{P_2}}^{\perp}$.If there exsits another isotropic line $\bar{Y} \in {\bar{P_1}}^{\perp} \cap {\bar{P_2}}^{\perp}$,we can get $\langle Y,P_1 \rangle=0,\langle Y,P_2 \rangle=0$.It means $\langle Y,P_1-t_0 P_2 \rangle=0$.From lemma 2.1,we know the lines in ${\bar{P_1}}^{\perp} \cap {\bar{P_2}}^{\perp}$ are all isotropic. ${\bar{P_1}}^{\perp} \cap {\bar{P_2}}^{\perp}$ is a plane,the number of lines on it is $q+1$,and $\bar{P_1},\bar{P_2}$ are anisotropic lines,from lemma 2.5,no matter $\alpha$ is a nonzero square element or a non-square element,the number of isotropic lines on $\bar{P_1}^{\perp}$ and $\bar{P_2}^{\perp}$ is $q+1$.So $\bar{P_1}^{\perp}={\bar{P_1}}^{\perp} \cap {\bar{P_2}}^{\perp}=\bar{P_2}^{\perp}$.$\bar{P_1}=\bar{P_2}$,contradiction.So $\overline{P_1-t_0 P_2}$ is the unique isotropic line on ${\bar{P_1}}^{\perp} \cap {\bar{P_2}}^{\perp}$.Lemma has proved.
\end{proof}
From lemma 2.2,the number of isotropic lines on a plane is $0$,$1$,$2$ or $q+1$.In lemma 3.2,we know if $\langle P_1,P_1\rangle\langle P_2,P_2\rangle\not=\langle P_1,P_2\rangle^2$,the number of isotropic lines on ${\bar{P_1}}^{\perp} \cap {\bar{P_2}}^{\perp}$ is even.

We know $N$ denote the set of anisotropic line in $P(V)$.We also divide $N$ into two parts, $S$ and $T$. $S=\{\bar{X}\in P(V)|Q(X)\  equals\  a\  nonzero\  square\  element\  in\  \mathbb{F}_q\}$, $T=\{\bar{X}\in P(V)|Q(X)\  equals\  a\  nonsquare\  element\  in\  \mathbb{F}_q\}$.The lines in $S$ are called square anisotropic lines and the lines in $T$ are called non-square anisotropic lines.$N$ is the disjoint union of $S$ and $T$.Any $\bar{X} \in S$,we know $Q(X)$ equals a nonzero square element.We can find unique $k_X$ such that $Q(k_X X)=1$, $k_X\in  \mathbb{F}_q^*$, $\bar{X}$ and $\overline{k_X X}$ are the same line.Similarly,if $\bar{Y} \in T$,we know $Q(X)$ equals a non-square element.We can find unique $k_Y$ such that $Q(k_Y Y)=\beta$, $k_Y\in  \mathbb{F}_q^*$,$\beta$ is a fixed non-square element in $\mathbb{F}_q^*$, $\bar{Y}$ and $\overline{k_Y Y}$ is the same line.So in the following discussion,if $\bar{X}$ denotes a square anisotropic line,then we choose $X$ such that $\langle X,X \rangle =1$.Similarly,if $\bar{Y}$ denotes a non-square anisotropic line,then we choose $Y$ such that $\langle Y,Y \rangle =\beta$.

\begin{lemma} $\bar{X}$ is a fixed square anisotropic line, such that $\langle X,X\rangle =1$. Then the number of  $\bar{Y}$ satisfy
\begin{equation}  \begin{cases}    \langle X,Y\rangle ^2=1 \\\langle Y,Y\rangle =1  \end{cases}\end{equation} is even ($\bar{Y}\not=\bar{X}$).
\end{lemma}
\begin{proof}The equation has two parts:
\begin{equation}\begin{cases}    \langle X,Y\rangle =1 \\\langle Y,Y\rangle =1  \end{cases}\end{equation}
 and \begin{equation}  \begin{cases}    \langle X,Y\rangle =-1 \\\langle Y,Y\rangle =1  \end{cases}
\end{equation}
If $Y$ is a solution of (2),then $-Y$ is a solution of (3).Similarly,if $Y$ is a solution of (3),then $-Y$ is a solution of (2).Consider $\bar{Y}$ and $\overline{-Y}$ are the same line,so we only solve equations (2).
$\langle Y-X,X\rangle =\langle Y,X\rangle -\langle X,X\rangle =0$, $\langle Y-X,Y-X\rangle =\langle Y,Y\rangle -2\langle Y,X\rangle +\langle X,X\rangle =0$.So $\overline{Y-X}$ is an isotropic line and $\overline{Y-X} \in {\bar{X}}^{\perp} $.From lemma 2.5,when $\bar{X}$ is an anisotropic line,there are $q+1$ isotropic lines on ${\bar{X}}^{\perp}$,no matter $\alpha$ is a nonzero square element or a non-square element.Let $\bar{X_0},\bar{X_1},\cdots,\bar{X_q}$ be the $q+1$ isotropic lines,then  $\overline{Y-X}=\bar{X_i}$, $i\in\{0,1,\cdots,q\}$.We get $Y=X+kX_i,k\in \mathbb{F}_{q}^*$,because  $\bar{Y}\not=\bar{X}$.We can verify $Y=X+kX_i$ are the solutions and different $\bar{Y}$ represent different lines.In conclusion,the number of $Y$ is $q^2-1$,even.
\end{proof}

\begin{lemma} $\bar{X}$, $\bar{Y}$ are fixed square anisotropic line, $\bar{X}\not=\bar{Y}$. $\langle X,X\rangle =\langle Y,Y\rangle =1$.Then the number of $\bar{Z}$ satisfy \begin{equation}  \begin{cases}    \langle X,Z\rangle ^2=1 \\\langle Y,Z\rangle ^2=1\\\langle Z,Z\rangle =1  \end{cases}\end{equation}
is even ($\bar{Z}\not=\bar{X},\bar{Y}$).
\end{lemma}
\begin{proof}
The equation has four parts.\begin{equation}  \begin{cases}    \langle X,Z\rangle =1 \\\langle Y,Z\rangle =1\\\langle Z,Z\rangle =1  \end{cases}\end{equation} and \begin{equation}  \begin{cases}    \langle X,Z\rangle =1 \\\langle Y,Z\rangle =-1\\\langle Z,Z\rangle =1  \end{cases}\end{equation} and \begin{equation}  \begin{cases}    \langle X,Z\rangle =-1 \\\langle Y,Z\rangle =1\\\langle Z,Z\rangle =1  \end{cases}\end{equation} and \begin{equation}  \begin{cases}    \langle X,Z\rangle =-1 \\\langle Y,Z\rangle =-1\\\langle Z,Z\rangle =1  \end{cases}\end{equation}Consider that the solution of (5)and(8) represent the same lines and the solution of (6)and(7)represent the same lines,so we only solve (5)and(6).It is clear (5)(6)have no same solution.Discuss blow:\\
case1. $\langle X,Y\rangle \not=\pm1$.First we find out nonzero constants $k_1=(-1-\langle X,Y\rangle )/(1-\langle X,Y\rangle )$, $k_2=2/(1-\langle X,Y\rangle )$,such that $k_1+k_2=1,k_1+\langle X,Y\rangle k_2=-1$. If $Z$ is a solution of (5),then $k_1 Z+k_2 X$ is a solution of (6),similarly,if $Z'$ is a solution of (6),then $(Z'-k_2 X)/{k_1}$ is a solution of (5).In conclusion,if the solutions of (5) or (6) exsit ,the number of solution of (4) is even.if (5)and(6) both have no solutions,the number is zero,also even.Case 1 has proved.\\
case 2. $\langle X,Y\rangle =1$.First we solve equation (5).We can get $\langle Z-X,X\rangle =0$, $\langle Z-X,Y\rangle =0$, $\langle Z-X,Z-X\rangle =0$. $Z\not=X$,so $\overline{Z-X}$ is an isotropic line and $\overline{Z-X} \in {{\bar{X}}^{\perp}\cap {\bar{Y}}^{\perp}}$.Because $\bar{X}$ and $\bar{Y}$ are different anisotropic lines,according to lemma 3.2 , ${{\bar{X}}^{\perp}\cap {\bar{Y}}^{\perp}}$ has unique isotropic line,let it be $\bar{W}$.Then $\overline{Z-X}=\bar{W}$,we got $Z-X=tW,t \in \mathbb{F}_{q}^*$, $q-1$ solutions.Notice that $Y-X$ satisfies $\overline{Y-X} \in {{\bar{X}}^{\perp}\cap {\bar{Y}}^{\perp}}$, $Y\not=X$,so $Y$ has form $Y=X+t_0 W$,$t_0\in\mathbb{F}_{q}^*$. $Y$ must be one of the $q-1$ solutions.Delete $Y$ from the $q-1$ solutions,we got $q-2$ solutions of (5).\\
Second we solve equation (6).We can also get $\langle Z-X,X\rangle =0$, $\langle Z-X,Z-X\rangle =0$,$Z\not=X$,so $\overline{Z-X}$ is an isotropic line on ${\bar{X}}^{\perp}$.No matter $\alpha$ is a nonzero square element or a non-square element,there are $q+1$ isotropic lines in ${\bar{X}}^{\perp}$.Let them be $\bar{X_0},\bar{X_1},\cdots,\bar{X_q}$.So $\overline{Z-X}=\bar{X_i}$, $i\in\{0,1,\cdots,q\}$. $Z$ has form $Z=X+kX_i$, $k\in\mathbb{F}_{q}^*$.Moreover,$Z$ satisfies $\langle Z,Y\rangle =-1$,it means $\langle X+kX_i,Y\rangle =-1$, $k\langle X_i,Y\rangle =-2$.Consider ${{\bar{X}}^{\perp}\cap {\bar{Y}}^{\perp}}$ has unique isotropic line,let it be $\bar{X_0}$.It means $\langle X_0,Y\rangle =0$, $\langle X_i,Y\rangle \not=0$, $i\in\{1,2,\cdots,q\}$.For $X_i$,$\langle X_i,Y\rangle \not=0$,there exists unique $k=k_{X_i}$ such that $k_{X_i}\langle X_i,Y\rangle =-2$.We solve $q$ solutions of (6).\\
In conclusion,in case 2,the number of solutions of (4) is $2q-2$,even.\\
case 3. $\langle X,Y\rangle =-1$.Let $X'=X,Y'=-Y,Z'=Z$.It's the same as case 2.
Lemma has proved.
\end{proof}

\begin{lemma} $\bar{X}$ is a fixed non-square anisotropic line,such that $\langle X,X\rangle =\beta$.The number of $\bar{Y}$ satisfy \begin{equation}  \begin{cases}    \langle X,Y\rangle ^2={\beta}^2 \\\langle Y,Y\rangle =\beta  \end{cases}\end{equation} is even ($\bar{Y}\not=\bar{X}$).
\end{lemma}
The proof is similar to lemma 3.3.

\begin{lemma} $\bar{X}$,$\bar{Y}$ are fixed non-square anisotropic line, $\bar{X}\not=\bar{Y}$. $\langle X,X\rangle =\langle Y,Y\rangle =\beta$.The number of $\bar{Z}$ satisfy \begin{equation}  \begin{cases}    \langle X,Z\rangle ^2={\beta}^2 \\\langle Y,Z\rangle ^2={\beta}^2\\\langle Z,Z\rangle =\beta  \end{cases}\end{equation} is even ($\bar{Z}\not=\bar{X},\bar{Y}$).
\end{lemma}
The proof is similar to lemma 3.4.\\

Now we can prove prop 3.1 . In the above discussion,we know $N=S
\cup T$.So we can give a partition of
$\mathbf{G}_{AA}$.\begin{equation} \left( \begin{array}{ccc}
\mathbf{G}_{SS}& \mathbf{G}_{ST}\\
\mathbf{G}_{TS}& \mathbf{G}_{TT}\\
\end{array} \right),
\end{equation}

According to \cite{1},we know the 4 submatrices are all square matrices.First we calculate $\mathbf{{G}}_{AA}^2$.
The element at position $(\bar{X},\bar{X})$ of $\mathbf{{G}}_{AA}^2$ equals the number of anisotropic lines on ${\bar{X}}^{\perp}$.No matter $\alpha$ is a nonzero square element or a non-square element,it is $q^2$ , odd.
The element at position $(\bar{X},\bar{Y})$ ($\bar{Y}\not=\bar{X}$)
of $\mathbf{{G}}_{AA}^2$ equals the number of anisotropic lines on
${\bar{X}}^{\perp} \cap {\bar{Y}}^{\perp}$.In particular,from lemma 3.2,it is odd
if and only if $\langle X,X\rangle\langle Y,Y\rangle=\langle
X,Y\rangle^2$.When $\bar{X} \in S,\bar{Y} \in T$,we know $\langle
X,X \rangle$ equals a square element and $\langle Y,Y \rangle$
equals a non-square element,so $\langle X,X\rangle\langle
Y,Y\rangle=\langle X,Y\rangle^2$ can't set up.
We obtain $\mathbf{{G}}_{AA}^2=I+ \left(
\begin{array}{ccc}
\mathbf{B}_{1}& \mathbf{O}\\
\mathbf{O}&\mathbf {B}_{2}\\
\end{array} \right)\pmod 2$,
The element on diagonal of $\mathbf{B}_{1}$ and $\mathbf{B}_{2}$ are all $0$.
Then $\mathbf{{G}}_{AA}^4=I+
\left( \begin{array}{ccc}
\mathbf{B}_{1}^2& \mathbf{O}\\
\mathbf{O}&\mathbf {B}_{2}^2\\
\end{array} \right)\pmod 2$. 

We want to prove $\mathbf{B}_{1}^2 \equiv 0 \pmod 2,\mathbf{B}_{2}^2 \equiv 0 \pmod 2$.
Calculate $\mathbf{B}_{1}^2$:
The element at position $(\bar{X},\bar{X})$:It equals the number of set $|\{\bar{Y}\not=\bar{X}|\langle X,X\rangle \langle Y,Y\rangle =\langle X,Y\rangle ^2,\bar{Y}\in N\}|$.According to lemma 3.3,it is even.The element at position $(\bar{X},\bar{Y})$ ($\bar{Y}\not=\bar{X}$):It equals the number of set
$|\{\bar{Z}\not=\bar{X},\bar{Y}|\langle Z,Z\rangle \langle X,X\rangle =\langle Z,X\rangle ^2,\langle Z,Z\rangle \langle Y,Y\rangle =\langle Z,Y\rangle ^2,\bar{Z}\in N\}|$.According to lemma 3.4,it is even.

Now we get $\mathbf{B}_{1}^2 \equiv 0\pmod 2$.It's similar to
$\mathbf{B}_{2}^2$,according to lemma 3.5 and lemma 3.6,we obtain
$\mathbf{B}_{2}^2 \equiv 0\pmod 2$.So,$\mathbf{{G}}_{AA}^4 \equiv I
\pmod 2$,no matter $\alpha$ is a nonzero square element or a
non-square element.It is clear the 2-rank of $\mathbf{{G}}_{AA}$ is
of full rank.Prop 3.1 has proved.


\begin{thebibliography}{20}
\bibitem{1}
 {Chunlei Liu,Yan Liu}, ¡°Incidence Matrices of  Finite Quadratic Spaces¡±,to appear.
\bibitem{2}
 {Kenneth Ireland, Michael Rosen}, ¡°A Classical Introduction to Modern Number Theory¡±, {\it Springer-Verlag}(1990).
\bibitem{3}
 {J.W.P.Hirschfeld}, ¡°Projective Geometries over Finite fields¡±, {\it Oxford University Press}(1998).
\bibitem{4}
{S. Droms, K. E. Mellinger, C. Meyer},¡°LDPC codes generated by conics in the classical projective plan¡±,
{\it Des. Codes Cryptogr.}, {\bf 40}~(2006), 343-356.
\bibitem{5}
{S. Peter, J. Wu, Q. Xiang},¡°Dimensions of some binary codes arising from a conic in $PG(2,q)$¡±,
{\it J. Combin. Theory Ser.A118}, {\bf 3}~(2011), 853-878.
\bibitem{6}
{Adonus L. Madison, J. Wu},¡°On binary codes from conics in $PG(2,q)$¡±,
{\it European J. Combin.} {\bf 33} (2012), 33-48.
\bibitem{7}
{J. Wu},¡°Proofs of two conjectures on the dimensions of binary codes¡±,
{\it J. Combin. Theory Ser.}.(2012),
\end{thebibliography}
\end{document}